 \documentclass{amsart}
 
 \usepackage{amssymb,amsmath}
\allowdisplaybreaks

\newfont{\cursive}{eusm10}

\newtheorem{theorem}{Theorem}[section]
\newtheorem{proposition}[theorem]{Proposition}
\newtheorem{lemma}[theorem]{Lemma}
\newtheorem{corollary}[theorem]{Corollary}
\theoremstyle{definition}
\newtheorem{definition}[theorem]{Definition}
\newtheorem{problem}{Problem}

\theoremstyle{remark}
\newtheorem{remark}[theorem]{Remark}

\numberwithin{equation}{section}

\newcounter{thlistctr}
\newenvironment{thlist}{\ 
\begin{list}%
{\alph{thlistctr}}%
{\setlength{\labelwidth}{2ex}%
\setlength{\labelsep}{1ex}%
\setlength{\leftmargin}{6ex}%
\usecounter{thlistctr}}}%
{\end{list}}

\newcommand{\whbQ}{\widehat{\bf Q}}
\newcommand{\whq}{\widehat{q}}

\newcommand{\bbN}{{\mathbb N}}
\newcommand{\bbQ}{{\mathbb Q}}
\newcommand{\bbR}{{\mathbb R}}

\newcommand{\sE}{{\sf E}}

\newcommand{\fG}{{\mathfrak G}}
\newcommand{\fQ}{{\mathfrak Q}}
\newcommand{\fS}{{\mathfrak S}}
\newcommand{ \fa }{{\mathfrak a }}
\newcommand{ \fb }{{\mathfrak b }}
\newcommand{ \fc }{{\mathfrak c }}
\newcommand{ \fd }{{\mathfrak d }}
\newcommand{ \fk }{{\mathfrak k }}

\newcommand{\bA}{{\bf A}}
\newcommand{\bB}{{\bf B}}
\newcommand{\bC}{{\bf C}}
\newcommand{\bF}{{\bf F}}

\newcommand{\bP}{{\bf P}}
\newcommand{\bQ}{{\bf Q}}
\newcommand{\bR}{{\bf R}}
\newcommand{\bS}{{\bf S}}
\newcommand{\bT}{{\bf T}}
\newcommand{\bU}{{\bf U}}

\newcommand{\bp}{{\bf p}}

\newcommand{\ro}{{\rm o}}

\newcommand{\cA}{{\mathcal A}}
\newcommand{\cB}{{\mathcal B}}

\newcommand{\cF}{{\mathcal F}}
\newcommand{\cD}{{\mathcal D}}

\newcommand{\cL}{{\mathcal L}}
\newcommand{\cM}{{\mathcal M}}
\newcommand{\cP}{{\mathcal P}}

\newcommand{\cT}{{\mathcal T}}
\newcommand{\cX}{{\mathcal X}}
\newcommand{\cU}{{\mathcal U}}

\newcommand{\exist}{\exists}
\newcommand{\univ}{\forall}
\newcommand {\lle}{<}
\newsymbol \ndiv   232D
\renewcommand{\mod}{{\sf mod\ }}
\newcommand{\MSO}{{\rm MSO}}
\newcommand{\RT}{{\sf RT}}
\newcommand{\RTOS}{{{\sf RT}_1}^\star}
\newcommand{\RTN}{{\sf RTN}}
\newcommand{\RTNOS}{{{\sf RTN}_1}^\star}
\newcommand{\bigexpp}[1]{\exp\big(#1\big)}

\newcommand{\fq}{\mathfrak{q}}

\newcommand{\fn}{\mathfrak{n}}

\newcommand{\MODULES}{{\sf MODULES}}

\usepackage[mathscr]{euscript}
\allowdisplaybreaks

\title[MSO 0--1 Laws]
{Monadic Second-Order Classes of Forests with a Monadic Second-Order  0--1 Law}

\author{Jason Bell, Stanley Burris, and Karen Yeats}

\date{\today}    

\begin{document}

\begin{abstract}
Let $\cT$ be a monadic-second order class of finite trees, and let $\bT(x)$ 
be its (ordinary) generating function, with radius of convergence $\rho$.
If $\rho \ge 1$ then $\cT$ has an explicit specification (without using 
recursion) in terms of the operations of union, sum, stack, and the
multiset operators $(n)$ and $(\ge n)$. 
Using this, one has an explicit expression for $\bT(x)$ in terms of the initial 
functions $x$ and $x\cdot \big(1-x^n\big)^{-1}$, the operations of addition and 
multiplication, and the P\'olya exponentiation operators $\sE_n, \sE_{\ge n}$.

Let $\cF$ be a monadic-second order class of finite forests, and let 
$\bF(x)=\sum_n f(n) x^n$ be its (ordinary) generating function.
Suppose $\cF$ is closed under extraction of component trees and sums of
forests.  
Using the above-mentioned structure theory for the class $\cT$ of trees in $\cF$,
Compton's theory of 0--1 laws, and a significantly strengthened version of 2003 
results of Bell and Burris on generating functions,  we show that  $\cF$ has a 
monadic second-order 0--1 law iff the radius of convergence of $\bF(x)$ is 1 iff the
 radius of convergence of $\bT(x)$ is $\ge 1$.
\end{abstract}

\maketitle

\section{Introduction}

In the late 1980s, Compton (\cite{Compton1987}, \cite{Compton1989}) 
introduced a new method to show that a class $\cA$ of finite relational 
structures has a monadic second-order (MSO) 0--1 law,\footnote
{See the Appendix for a discussion of monadic second-order logic.
Given a class $\cA$ and a logic $\cL$, we say ``$\cA$ has a $\cL$ 0--1 law'' if, 
for any $\cL$-sentence $\varphi$,  the class $\cA_\varphi$, of structures in $\cA$ 
for which $\varphi$ is true, has asymptotic density 0 or 1 in $\cA$. 
 (See $\S$\ref{sec 4.1}.)}
 a method that depended only on a property of the generating function 
 $\bA(x)$ for  $\cA$, and not on the nature of the structures in the class. 
 The pre-condition imposed on the class $\cA$ was that it be closed under 
 the extraction of components and sums of its members---we say such a 
 class is {\em adequate} (or {\em Compton-admissible}). 
 Compton analyzed both labelled and unlabelled classes---this paper concerns
  {\em unlabelled} classes $\cA$ and their associated {\em ordinary}
  generating functions $\bA(x) = \sum a(n) x^n$.
 
 \begin{theorem}[Compton] \label{Compton's thm}
Let $d$ be the gcd of the sizes of the members of an adequate class $\cA$ of 
relational structures, and let $\rho_\bA$ be the radius of convergence of the 
generating function $\bA(x)$ of $\cA$.  
\begin{thlist}
\item
If $\cA$ has a first-order 0--1 law then $\rho_\bA \in \{0,1\}$.
\item
If $\rho_\bA > 0$, then $\cA$ has a monadic second-order 0--1 law iff it has a 
first-order 0--1 law iff 
\begin{equation}  \label{Compton's test}
	 \lim_{n\rightarrow \infty} \frac{a\big((n-1)d\big)}{a(nd)} \ =\  1.
 \end{equation}
 \end{thlist}
 \end{theorem}

This paper is about adequate classes of forests---for classes $\cF$ of
forests, the radius of convergence $\rho_\bF$ of the generating function 
$\bF(x) = \sum_n f(n)x^n$ is always positive. Thus Compton's Theorem
on 0--1 laws is slightly simpler in the case of forests.

\begin{corollary}\label{forest thm}
Given an adequate class $\cF$ of forests, let $d$ be the gcd of the sizes of the 
forests in $\cF$. Then $\cF$ has a monadic second-order 0--1 law iff it has a 
first-order 0--1 law iff 
\begin{equation}\nonumber
	 \lim_{n\rightarrow \infty} \frac{f\big((n-1)d\big)}{f(nd)} \ =\  1.
 \end{equation}
\end{corollary}

Common examples of adequate classes of forests are usually MSO-classes, 
that is, they are defined by a MSO-sentence. 
For example, Compton applied his theorem to two adequate classes of forests, 
namely to (1) forests of trees of height 1, and (2) forests of linear trees---these
are clearly MSO classes.
 An adequate class $\cF$ of forests is determined by its subclass $\cT$ of trees, 
 and  $\cF$ is a MSO-class iff $\cT$ is a MSO-class. MSO-classes of trees include 
 most of the basic examples of classes of trees one finds in the literature, for example, trees  of bounded height, chains, trees of bounded width,  binary trees, etc.
  
 Although condition \eqref{Compton's test} is exceedingly simple to state, it can be 
 challenging to verify that it actually holds for a given adequate class $\cA$. 
 Most of the practical success in this direction has been in finding conditions on the 
 generating function $\bP(x) = \sum_n p(n)x^n$ of the class $\cP$ of components of 
 $\cA$, conditions which ensure $\cA$ satisfies \eqref{Compton's test}. 
 Notable results are: (1) Bell's theorem (\cite{Bell2002}, 2002), which says that 
 polynomially bounded growth of  the component count function $p(n)$ is 
 sufficient,\footnote
  {An analog of Bell's polynomial bound theorem was proved for labelled structures 
  in  2008 by Burris and Yeats \cite{BuYe2008}.}
and (2) Bell and Burris's  theorem (\cite{BeBu2003}, 2003), which says that
 $p\big((n-1)d\big)/p(nd) \rightarrow 1$ is sufficient. 
 Although these results have greatly increased the scope of application of Compton's
 theory,  finding a comprehensive practical approach to verifying
 \eqref{Compton's test} remains a vexing problem.
 
This paper provides a transparent description of when an 
adequate MSO-class $\cF$ of forests has a MSO 0--1 law. 
Let $\cT$ be the class of trees in $\cF$. 
Then  {\em $\cF$ has a MSO 0--1 law iff $\rho_\bF=1$ iff $\rho_\bT\ge 1$ iff $\cT$
 has an explicit specification in terms of four natural operations on classes of trees.}

 The proof involves three key steps:
  First a structure theorem is established which shows that a MSO-class of trees  
  $\cT$ with $\rho_\bT \ge 1$ has an explicit (non-recursive) specification.
 Secondly, using this result, a description of the generating function $\bT(x)$ of 
 $\cT$ is determined.
Thirdly, the arguments used in the above-mentioned \cite{BeBu2003} are 
 thoroughly reworked  to cover the generating functions described in the second step.

\section{Preliminaries}
 The \textit{radius} of a class of forests is the radius of convergence of its ordinary 
 generating function. 
 A class of forests must have its radius in $[0,1]\cup\{\infty\}$ since the coefficients 
 of the generating function are nonnegative integers. 
 The classes with radius infinity are precisely the finite classes (whose generating 
 functions are polynomials). 
Compton's test for a MSO 0--1 law for a class $\cF$ of forests requires the radius of 
the class to equal $1$, and thus the radius of the class $\cT$ of trees in $\cF$ must 
be $\ge 1$.

Forests of rooted trees occur in two basic incarnations, namely as acyclic graphs
and posets. 
 The choice of fundamental language, whether that of graphs (with a binary edge 
 relation $E$ and a unary root relation $R$), or that of posets (with a less than 
 relation $<$), is not significant---being definable by a MSO sentence is a robust 
 concept that is not affected by the choice of the basic relation(s).  
In this paper forests are posets $(F,<)$.

Capital boldface letters will be used for power series, and the corresponding
lowercase letters for the coefficients. For example,
$$
	\bA(x)\ =\ \sum_{n=0}^\infty a(n)x^n. 
$$
Given a class $\cT$ of trees, $\bT(x):=\sum_{n\ge 1} t(n)x^n$ is used for its generating
function, where $t(n)$ counts the number of trees of size $n$ in $\cT$.
Likewise $\bF(x):= \sum_{n\ge 1}f(n)x^n$ is the generating function for a class $\cF$ 
of forests.
It will be usual to abbreviate a forest $(F,\lle$) simply as $F$; and likewise a tree 
$(T,\lle)$ as $T$. 
The one-element tree is $\bullet$;  it is also the only one-element forest.

\section{Four Class Constructions} \label{constr sec}
This section describes the four natural constructions---union, sum, multiset 
and stack---that will be used to give a transparent description of a 
MSO-class of trees of radius $\ge 1$.
 Namely such a class is a composition of these constructions applied to the initial 
 object, which is the singleton class of the one-element tree.
The first three constructions, \emph{union}, \emph{sum} and 
\emph{multiset}, are well-known and can be used with any class 
of purely relational structures (such as graphs, posets, etc.).

\subsection{The Union Construction} \label{union constr}

Given classes $\cF_1,\ldots,\cF_m$ of forests, the \emph{union} operation 
$\bigcup_i \cF_i$ is just as one would expect:
$$
	\bigcup_{i=1}^m \cF_i\ =\ \big\{F:F\in \cF_i\ \text{for some } i\big\}. 
$$
If the classes $\cF_1,\ldots,\cF_m$ are pairwise disjoint then 
$$
	\bF(x)\ =\ \sum_{i=1}^m \bF_i(x). 
$$
%

%
\subsection{The Sum Construction} \label{sum constr}
Given trees $T_1,\ldots,T_m$,  the \emph{sum} operation 
$$
	\sum_{i=1}^m T_i,\quad \text{or}\quad \ T_1+\cdots+ T_m, 
$$
is defined by taking a \emph{disjoint union} of the $T_i$. 
This means we assume we have renamed the elements of the trees $(T_i,<_i)$ 
so they are pairwise disjoint, and then we form the forest
$(T,<):= \Big(\bigcup_i T_i, \bigcup_i <_i\Big)$.
For classes $\cT_1,\ldots,\cT_m$ of trees, the \emph{sum} of the classes 
is given by\footnote{
Note, the union operation gives the union \emph{of classes} of trees, whereas the sum operation gives the class
\emph{whose members} are disjoint unions of trees from the classes.}
$$
	\sum_{i=1}^m \cT_i\:=\ \big\{T_1+\cdots+ T_m:T_i\in \cT_i\big\}. 
$$

If the classes $\cT_1,\ldots,\cT_m$ are pairwise {\em disjoint} then we speak of 
a {\em disjoint} sum $\cF:=\sum_i\cT_i$, and in this case the generating function is
$$
	\bF(x)\ =\ \prod_{i=1}^m \bT_i(x). 
$$
This follows from the fact that every forest $F$ has a unique decomposition into 
a disjoint sum of trees.

%
\subsection{The Multiset Constructions} \label{scalar mult constr}
Given a class $\cT$ of trees, $\gamma\cT$ is the class of forests where each 
member $F$ is a sum of $\gamma$ many copies of trees from $\cT$ 
(allowing repeats). 
The two forms for $\gamma$ that we use are $m$ and $\ge m$, where $m\ge 0$\,:
\begin{eqnarray*}
	m\cT &:=& \sum_{i=1}^m \cT\\
	(\ge m)\cT &:=& \bigcup_{n\ge m} n\cT. 
\end{eqnarray*}

The generating function $[\gamma\cT](x)$ for $\gamma\cT$ is easily derived 
from the generating function for $\cT$ using the following operators that act on 
power series:\footnote{The function $\sE_m$ comes from 
the well-known formula for the generating function 
for the set of objects that can be expressed as a
sum of exactly $m$ components (see Appendix B of \cite{density}).}

\begin{eqnarray*}
	\sE_0\big(\bA(x)\big)&:=& 1\\
	\sE_m\big(\bA(x)\big)
&:=& 
	\sum_{j=1}^m\frac{1}{j!}\sum_{\substack{m_1+\cdots+m_j = m\\m_i\geq 1\hfill}}
	\frac{1}{m_1\cdots m_j}\cdot \bA(x^{m_1})\cdots \bA(x^{m_j}), \quad m>0\\
	\sE_{(\ge m)}\big(\bA(x)\big)
&:=&
	\sum_{j=m}^\infty \sE_j\big(\bA(x)\big), \quad m\ge 0. 
\end{eqnarray*}
We often abbreviate $\sE_{\ge 0}$ to $\sE$. 
For $\gamma$ any coefficient we have
\begin{equation}\label{gamma count}
	[\gamma \cT](x) \ =\ \sE_\gamma\big(\bT(x)\big). 
\end{equation}

An {\em adequate} class of forests $\cF$ is one of the form $(\ge 1)\cT$, that is, 
it consists of all the forests that can be formed using the trees from $\cT$.
Adequate classes of forests are precisely the classes of forests that are closed
under sum and the extraction of component trees.

%
\subsection{The Stack Construction}\label{stack sec}

As already mentioned, the previous three constructions are general purpose 
constructions that one can use with any classes of relational structures. 
However the stack construction described in this section has been specially 
designed for the study of trees.

Given a tree $T$ and a node $\nu$ in the tree, $T[\nu]$ is the {\em full subtree}
of $T$ rooted at $\nu$, consisting of all the elements of $T$ that are $\le \nu$.
Given another tree $T_1$, $T[\nu/T_1]$ is the tree obtained by replacing $T[\nu]$
in $T$ by $T_1$.

A {\em (construction) module}  $M = (T,\lambda)$ is a tree $T$ with a designated 
leaf $\lambda$.  The one-element module is called $1_M$.
The {\em stack construction} can be applied to a pair of modules or to a module 
and a tree. 
The stack $M_1\circ M_2$, where $M_i = (T_i,\lambda_i)$, is the module
$(T_1[\lambda_1/T_2], \lambda_2)$. 
The stack $M_1\circ T_2$, where $M_1=(T_1,\lambda_1)$, is the tree 
$T_1[\lambda_1/T_2]$.

Let  $\MODULES$ denote the class of modules. 
Then $(\MODULES,\circ,1_M)$ is a monoid (since the stack operation
 is associative). 
A module $M =(T,\lambda)$ is indecomposable in this monoid iff $\lambda$
is a node immediately below the root of $T$. 
Furthermore, since there is a unique maximal chain going from the root of 
$T$ to $\lambda$, it follows that the monoid of modules has the unique 
factorization property, and thus the cancellation property. 
This implies the monoid of modules is actually a free monoid, freely generated 
by its indecomposable members.

Stacking $n$ copies of a given module $M:= (T,\lambda)$ gives $M^n$, where
$M^0 = 1_M$. 
Let $M^{\ge 0}:= \bigcup_{n\ge 0} M^n$.
 The size $|M|$ of a module $M:= (T,\lambda)$ is defined to be  $|T|-1$, one less 
 than the size of the tree in the module. 
 Thus we have
\begin{eqnarray*}
	\big|1_M\big|&=& 0\\
	\big| M_1\circ M_2 \big|&=&\big|M_1\big| + \big|M_1\big|\\
	\big| M\circ T \big|&=& \big|M\big| + \big|T\big|\\
	\big| M^n\big|&=& n\cdot \big|M\big|.
\end{eqnarray*}

We can view stack as a class operation:
\begin{eqnarray*}
	\cM_1\circ \cM_2\:=\ \big\{M_1\circ M_2: M_i \in \cM_i\big\}\\
	\cM\circ \cT \:=\ \big\{M\circ T: M \in \cM, T \in \cT\big\},
\end{eqnarray*}
with the special cases $M\circ \cM$, $M\circ \cT$, etc., where one of the classes 
is a singleton.

Given a tree $T$ and a chain of nodes $\nu_0>\nu_1>\ldots> \nu_{\fk}$, with 
$\nu_0$ the root of the tree, one has a decomposition of the tree as a stack
\begin{equation} \label{decomp}
	T\ =\ M_0 \circ  M_1\circ \cdots \circ  M_{\fk-1}\circ T_\fk
\end{equation}
where $M_i:= (T_i,\nu_{i+1})$ with 
$T_i:= T[\nu_i] \smallsetminus \big(T[\nu_{i+1}] \smallsetminus \nu_{i+1}\big)$, 
for $0\le i< \fk$, and $T_\fk:= T[\nu_\fk]$. 
If $\nu_0>\nu_1>\ldots> \nu_\fk$ is a maximal chain then \eqref{decomp} 
is a {\em complete} stack decomposition of $T$.
%

%
\subsection{Compton's Equations and the Dependency Digraph}
Let $\cF$ be a MSO-class of forests. Then the subclass $\cT$ of trees in $\cF$ 
is also a MSO-class. 
Using Ehrenfeucht-Fra\"isse  games, in 1986 Compton noted that every MSO 
class of trees has an equational specification.\footnote
{This result was first published by Woods \cite{Woods1997} in 1997, 
with credits to Compton.}
 To describe this we need one more definition, namely if $F$ is a forest let 
 $\bullet \big/ F$  denote the tree obtained by adding a new element to $F$ 
 that is greater than all elements in $F$. 
 Then, for $\cF$ a class of forests, let 
$\bullet \big/ \cF:= \big\{\bullet \big/ F: F\in\cF\big\}$.

A MSO-class defined by a MSO-sentence of quantifier depth at most 
$\fq$ is called a $\MSO^\fq$-class.
%
\begin{proposition}[Compton] \label{Compton eq}
Given a positive integer $\fq$, let $\cT_0,\ldots, \cT_\fn$ be the partition of 
the class of all trees into minimal $\MSO^\fq$-classes, where $\cT_0$
has the one-element tree as its only member. 
Then one has the following: 
\begin{thlist}
\item
Any $\MSO^\fq$-class $\cT$ is a union of some of the $\cT_i$.
\item
There are finite sets $\Gamma_i$  whose members $\pmb{\gamma}$ are sequences 
$\gamma_0,\ldots,\gamma_\fn$ of coefficients, each of the form $m$ or $\ge m$,
such that one has the following system $\Sigma^\fq$ of equations providing a
specification for the classes $\cT_i$:
\begin{eqnarray*}
	\cT_0&=& \{\bullet\}\\
	\cT_i
	&=&
	 \bullet\Big /  \bigcup_{\pmb{\gamma}\in \Gamma_i} 
	 \sum_{j=0}^\fn \gamma_j \cT_j 
	 \quad \text{for } 1\le i\le \fn. 
\end{eqnarray*}
\item
From (b) one has the generating functions $\bT_i(x)$ for the $\cT_i$ 
defined by the system of equations:
\begin{eqnarray*}
	\bT_0(x)  &=&  x\\
	\bT_i(x)
&=& 
	x\cdot \sum_{\pmb{\gamma}\in \Gamma_i}
	 \prod_{j=0}^\fn \big[\gamma_j \bT_j\big] (x) 
	\quad \text{for } 1\le i\le \fn. 
\end{eqnarray*}
\end{thlist}
\end{proposition}

Given a system $\Sigma^\fq$ of Compton equations, let $\rho_i$ be the 
radius of the class $\cT_i$, $0\le i \le \fn$.
The {\em dependency digraph} for the system is 
$\cD_\fq = (\{0,1,\ldots,\fn\}, \rightarrow)$, 
where $i \rightarrow j$ means that some $\pmb{\gamma}\in\Gamma_i$ 
is such that $\gamma_j \neq 0$. 
$\rightarrow^+$ is the transitive closure of $\rightarrow$. 
Note, by Proposition \ref{Compton eq}(c), that $i\rightarrow j$, and hence 
$i\rightarrow^+ j$, implies $\rho_i \le \rho_j$.
The {\em strong component} of an element $i$ of  the dependency digraph is
$$
	[i] \ :=\ \big\{j \in D_\fq: i \rightarrow^+ j \rightarrow^+ i\big\}.
$$
Thus  $j\in [i]$ implies $\rho_i = \rho_j$.

If $[i]\neq \O$ then $\cT_i$ is an infinite class of trees and $\rho_i \le 1$.

We say $i > j$ in $\cD_\fq$ if $i \rightarrow^+ j$, but not conversely.

Define the {\em rank} of $i \in \cD_\fq$ to be its height in the poset 
$(\{0,1,\ldots,\fn\}, >)$.

%
\subsection{The modules $M_{ij}$}
Given a system $\Sigma^\fq$ of Compton equations, for $0 \le i,j \le \fn$ let 
$$
	\cM_{ij} \:=\ \big\{ M \in \MODULES: M\circ \cT_j \subseteq \cT_i\big\}.
$$
%
\begin{lemma}  \label{Mii lem}
Suppose $[i] \neq \O$ and $\rho_i = 1$. 
\begin{thlist}
\item
There is a unique module $M_{ii}\in \cM_{ii}$ such that $\cM_{ii}\ =\ (M_{ii})^{\ge 0} $.
\item
For each $j\in [i]$  there is a unique module $\widehat{M}_{ij}$ such that
$\cM_{ij}\ =\ \widehat{M}_{ij}\circ \cM_{jj} = \cM_{ii} \circ \widehat{M}_{ij}.$
\item
$([i],\rightarrow)$ is a directed cycle.
\end{thlist}
\end{lemma}
\begin{proof}
If (a) fails then there are two modules $M_1$ and $M_2$ in $\cM_{ii}$ such that 
neither has a proper stack factorization by modules in $\cM_{ii}$. 
Define $\cM:= \big\{ M_1,M_2\big\}$, $m:= \max(|M_1|, |M_2|)$, and let $T\in \cT_i$.
Then $\cM^{m+n+1}\circ T$ is a subset of $\cT_i$, and, by examining complete
decompositions using chains of maximum length, one sees that $\cT_i$ has at least 
$2^n$ trees of size at most $m(m+n+1) + |T|$. 
This contradicts the assumption that $\rho_i = 1$.

For item (b),  first it is clear that $\widehat{M}_{ii} = 1_M$.
For $i\neq j$, note that  $\cM_{ji}\circ \cM_{ij} \subseteq \cM_{jj}$.
Let $M\in \cM_{ij}$, $N\in \cM_{ji}$. Then $N\circ M \in \cM_{jj}$, so from (a)
there is an integer $n\ge 1$ such that $N\circ M = (M_{jj})^n$. 
By unique factorization there are unique modules $M'$ and $N'$ and integers 
$a,b\ge 0$ such that $N = (M_{jj})^{a}\circ N'$, $M  = M'\circ (M_{jj})^b$, and 
$N'\circ M' = M_{jj}$.
Holding $N$ fixed, we see from the last equation that $M'$ must be the same 
for all $M\in \cM_{ij}$. 
Thus $\cM_{ij} = M'\circ \cM_{jj}$. 
Using unique factorization once again, we see that only one member of $\cM_{ij}$ 
can fulfill the role of $M'$.

Item (c) follows from an argument like that used for (a).
Two distinct minimal paths from $j$ to $j$, for any $j \in [i]$, would lead to $\rho_j < 1$, 
which would contradict the fact that all $\rho_k$ are equal, for $k\in[i]$.
\end{proof}

%
\subsection{Explicit descriptions}

%
\begin{proposition} \label{structure thm}
Let $\Sigma^\fq$ be a system of Compton equations, and suppose $\cT_i$ 
has radius  $\ge 1$.
Then one has the following description of $\cT_i$ in terms of the $\cT_j$ with 
$j$ of smaller rank, for $1\le i\le \fn$. 
This leads to a corresponding expression for $\bT_i(x)$.
\begin{thlist}
\item
Suppose $[i]= \O$.  
Then the Compton equations give a description of $\cT_i$ in terms of the 
$\cT_j$ with $j$ of smaller rank; and likewise for the $\bT_i(x)$.
\item
Suppose $[i]\neq \O$. Let
$$
	\Gamma_i^0\:=\ \big\{ \pmb{\gamma}\in \Gamma_i: \gamma_j = 0 
	\text{ for all } j\in [i]\big\}.
$$ 
Then
\begin{eqnarray*}
	\cT_i&=& \bigcup_{k\in [i]} (M_{ii})^{\ge 0} \circ \widehat{M}_{ik}\circ 
	\Big( \bullet \Big / \bigcup_{\pmb{\gamma}\in \Gamma_k^0} 
	\sum_{j=0}^\fn \gamma_j \cT_j\Big)\\
	\bT_i(x)&=& \sum_{k\in[i]} \frac{x^{|\widehat{M}_{ik}| +1}}{1 -x^{|M_{ii}|}} 
	  \cdot  \sum_{\pmb{\gamma}\in \Gamma_k^0} \prod_{j=0}^\fn 
	 \sE_{\gamma_j}\big( \bT_j(x)\big).
\end{eqnarray*}
\end{thlist}
\end{proposition}
\begin{proof}
If $[i]=\O$ the result is clear.

For $[i]\neq \O$ one has $\rho_i \le 1$ since $\cT_i$ is infinite, so we can apply 
Lemma \ref{Mii lem}.
 Let $T$ be a tree in $\cT_i$, and let  $\nu_0 > \nu_1 > \cdots \nu_\fk$ be a 
 maximal chain in $T$. 
 Define $w: \{0,\ldots,\fk\} \rightarrow \{1,\ldots,\fn\}$, a map from the indices of the 
 nodes $\nu_i$ to the dependency digraph, by $w(x) = y$ if $T[\nu_x] \in \cT_y$. 
As $x$ moves from $0$ to $\fk$, the image $w(x)$ moves $m$ times around the 
directed cycle $([i],\rightarrow)$, from $i$ to $i$ each time, for some $m \ge 0$, 
and then either immediately exits the cycle or makes a partial trip around the cycle 
to some node and then exits the cycle.  Let $\varepsilon(i)$ be the node of the cycle
from which $w(x)$ exits the cycle.
 
 Let $r$ be the number of elements in $[i]$; and let $s =0$ if $i=\varepsilon(i)$, otherwise let
 $s$ be the length of the shortest directed path in the directed cycle from $i$ to $\varepsilon(i)$.
 Then $w(0) = w(r)=\ldots = w(mr) = i$ and $w(mr+s) = \varepsilon(i)$. 
 The stack of the first $mr+s$ modules  in the stack decomposition of $T$ derived from
 $\nu_0 > \cdots >  \nu_\fk$ gives the module $(M_{ii})^m\circ \widehat{M}_{i\varepsilon(i)}$. 
  We can assume the chain of nodes $\nu_x$ was chosen so that the number of 
  elements of $[i]$ in the range of $w$ is maximum.
  This means that $T[\nu_{mr+s}] \in \cT_{\varepsilon(i)}^0$, where
  $$
   \cT_{\varepsilon(i)}^0 \ =\  \bullet \Big / \bigcup_{\pmb{\gamma}\in \Gamma_{\varepsilon(i)}^0} 
	\sum_{j=0}^\fn \gamma_j \cT_j .
  $$
  ($\cT_{\varepsilon(i)}^0$ is the class of trees in $\cT_{\varepsilon(i)}$ whose proper subtrees are not
  in $\cT_\ell$ for any $\ell \in [i]$.)
  
  Putting these facts together, we have
  $$
 	 T \in (M_{ii})^m\circ \widehat{M}_{i{\varepsilon(i)}}\circ T[\nu_{mr+s}],
  $$
  leading to the description of $\cT_i$ in (b). 
  The translation into an expression for $\bT(x)$ is straightforward.
\end{proof}

\subsection{Definition of $\fG$}
Define $ \fG$  to be the closure of the class 
$$
	\{x\} \cup \Big\{ \frac{x}{1-x^m}: m\ge 1\Big\}
$$
under the operations
$$
	+,\ \times,\  \sE_m,\ \sE_{\ge m}\qquad\text{for }\ m\ge 1. 
$$

\begin{corollary} \label{in GEN}
\begin{thlist}
\item
Every MSO-class $\cT$ of trees of radius $\ge 1$ has its generating function in $\fG$.
\item
Every function $\bA(x)\in \fG$ has radius of convergence $\ge 1$.
\end{thlist}
\end{corollary}
\begin{proof}
For (a), use induction on the rank of $i$ to prove this for each $\cT_i$,  in view 
of Proposition \ref{structure thm}. 
Then use the fact that $\cT$ is a union of some of the $\cT_i$, and $\fG$ is 
closed under addition.
 
 For (b), note that the base functions have radius of convergence $\ge 1$,
 and applying the operations and operators preserves this property.
 \end{proof}

In order to understand the behavior of generating functions in $\fG$, 
we examine a larger class $\fS$.

%
\section{The Class $\fS$  of Power Series}

To define the class $\fS$ we need the notion of $\RTOS$.

%
\subsection{The classes $\RT_1$, $\RTOS$} \label{sec 4.1}
Let $\cA$ be an adequate class of relational structures, and let $\cP$ be the
class of components of $\cA$. 
The ordinary generating functions $\bP(x)$ and $\bA(x)$ are related by 
the partition identity:
\begin{equation}\label{part id}
	1\, +\, \sum_{n=1}^\infty a(n)x^n\ =\ \prod_{n=1}^\infty \big(1-x^n\big)^{-p(n)}. 
\end{equation}
Let $\bbN$ be the set of nonnegative integers.
The \emph{period} of $\cA$ is
$$
	d\ :=\ \gcd\big\{n: p(n)>0\big\}\ =\ \gcd\big\{n: a(n)>0\big\}. 
$$
$a(n)$ is 0 if $d\ndiv n$,  and it is eventually positive on the set $d\cdot\bbN$ 
of multiples of $d$.  
We say that a subclass $\cB$ of $\cA$ has an \emph{asymptotic density in $\cA$}
 if $b(nd)/a(nd)$ converges as $n\rightarrow \infty$\, (in which case the asymptotic 
 density of $\cB$ is the limiting value of the quotient). 

If the period of $\cA$ is 1, that is, $a(n)$ is eventually positive, then Compton's test 
is simply
\begin{equation}\label{RT}
	\lim_{n\rightarrow \infty} \frac{a(n-1)}{a(n)}\ =\ 1. 
\end{equation}
Let
\begin{eqnarray*}
	\RT_1
&:=& 
	\Big\{\bR(x)\in\bbR[[x]]\,:\, r(n)\succ 0\text{ and }
	\lim_{n\rightarrow \infty} \frac{r(n-1)}{r(n)}\ =\ 1\Big\}\\
	\RT_1[d] 
&:=&  
	\Big\{ \bR\big(x^d\big): \bR(x)\in\RT_1\Big\}, \quad \text{for }d\in\bbN\smallsetminus \{0\}\\
	\RTOS 
&:=& 
	\bigcup_{d\in\bbN\setminus\{0\}}  \RT_1[d], 
\end{eqnarray*}
where 
$r(n)\succ 0$ means $r(n)$ \textit{is eventually greater than} $0$.
It will be convenient to write $\cA\in\RT_1$, resp. $\cA\in\RTOS$, 
if $\bA(x)\in\RT_1$, respectively $\bA(x)\in\RTOS$. 
Likewise $r(n)\in\RT_1$ means $\bR(x)\in\RT_1$.

Define the classes $\RTN_1$, $\RTN_1[d]$ and $\RTNOS$ 
by intersecting the classes
$\RT_1$, $\RT_1[d]$ and $\RTOS$ with $\bbN[[x]]$.

Using this notation, Compton's condition \eqref{Compton's test} can be stated
as $\bA(x)\in\RTOS$.
The main result in \cite{BeBu2003} to prove  MSO 0--1 laws was 
 $\sE\big(\RTNOS\big) \subseteq \RTOS$.
To prove our main theorem, Theorem \ref{main thm}, we need the much 
stronger result $\sE\big(\fS \big)\subseteq \RTOS$
 which is stated in  Proposition 
\ref{main RT results} (the class $\fS$ is defined in $\S\ref{class S}$).

\subsection{Basic results about $\RT_1$ and  $\RTOS$}

The Cauchy product $\bC(x)$ of two power series $\bA(x)$ and $\bB(x)$  is defined by 
$$
	c(n)\ :=\ \sum_{j=0}^n a(j)\cdot b(n-j). 
$$ 
\begin{lemma} \label{closure in RT}
\begin{thlist}
\item
$\RT_1$ is closed under addition, multiplication by positive reals,  Cauchy product,
and\, asymptotic equality.
Furthermore,
$\bS(x)\in \RT_1$\, iff\, $\bS'(x)\in\RT_1$\, iff\, $x\bS(x)\in\RT_1$.
\item
For $d\in \bbN \smallsetminus \{0\}$, 
$\RT_1[d]$ is closed under addition, multiplication by positive reals, and Cauchy product.  
\item
$\RTNOS$ is closed under multiplication by positive integers as well as Cauchy product.
\end{thlist}
\end{lemma}

\begin{proof}
For (a), suppose $\bA(x), \bB(x) \in\RT_1$, and $r$ is a positive real. Clearly
$r\bA(x) \in \RT_1$. For the other conditions, note that $\bR(x) \in \RT_1$ iff
$r(n) \succ 0$ and $r(n) - r(n-1) = \ro\big(r(n)\big)$. Then to show $\RT_1$ 
is closed under addition use
\begin{eqnarray*}
\big(a(n) + b(n)\big) - \big(a(n-1) + b(n-1)\big)
&=&
\big( a(n) - a(n-1)\big) + \big( b(n) -b(n-1)\big)\\
&=&
\ro\big(a(n)\big) + \ro\big(b(n)\big) \\
&=&
\ro\big( a(n) + b(n) \big),
\end{eqnarray*}
as well as noting that $a(n) + b(n) \succ 0$.
To show closure under Cauchy product we have
\begin{eqnarray*}
a(n)b(n) - a(n-1)b(n-1)
&=&
a(n)\big(b(n) - b(n-1)\big) + \big( a(n) - a(n-1) \big) b(n-1)\\
&=&
a(n) \ro\big(b(n)\big) + \ro\big(a(n)\big) \Big(b(n) + \ro\big(b(n)\big)\Big)\\
&=&
\ro\big(a(n) b(n)\big),
\end{eqnarray*}
along with $a(n)b(n) \succ 0$. For asymptotic equality let $c(n) \sim a(n)$.
Then $c(n)\succ 0$, and 
$$
\frac{c(n-1)}{c(n)} \ \sim\ \frac{a(n-1)}{a(n)}\ \rightarrow \ 1
\quad\text{as }n \rightarrow \infty.
$$
 (b) follows from (a), and (c) is
Lemma 16.2 of \cite{Burris1997}.
\end{proof}

\begin{lemma} [Schur's Tauberian Theorem]   \label{Schur}
Suppose that $\bB(x)$ and $\bC(x)$ are power series such that
\begin{thlist}
\item
$\bB(x)$ has radius of convergence greater than 1,
\item
$\bB(1)>0$, and
\item
$\bC(x)\in \RT_1$.
\end{thlist}
Let\enspace $\bA(x)\ :=\ \bB(x)\cdot \bC(x)$.  Then 
$$
	a(n)\ \sim\  \bB(1)\cdot c(n). 
$$
\end{lemma}

\begin{proof}
(See $\S$3.8 of \cite{density}.)
\end{proof} 

\begin{corollary} \label{schur cor}
Suppose  that $\bB(x)$ and $\bC(x)$ are power series such that the radius 
of convergence of\enspace $\bB(x)$ is greater than 1, and $\bC(x)\in\RT_1$. 
If\enspace $\bB(1)> 0$  then 
$$
	\bB(x)\cdot \bC(x)\,\in\, \RT_1. 
$$
\end{corollary}

\begin{lemma} \label{var on schur}
Suppose $\bA(x) = \bB(x)\cdot \bC(x)$, where $\bA(x)$ and $\bB(x)$ 
have nonnegative coefficients and $\bC(x)$ is not the zero power series.
If
\begin{thlist}
 \item
$\displaystyle \frac{a(n-1)}{a(n)}\ \preccurlyeq\  1$, 
\item
 $b(n)=\ro\big(a(n)\big)$, and
 \item
$\bC(x)\in\RT_1\,$, 
\end{thlist}
then $\bA(x)\in\RT_1$.  
\end{lemma} 

\begin{proof}
This is Lemma 3.3 of \cite{BeBu2003}.
\end{proof}

\begin{lemma} \label{oSum}
Let $\bA(x)\in\bbN[[x]]$ have coefficients that are eventually positive, and
suppose that for some $u\ge 0$, 
$$
	a(n) - a(n-1)\ =\ \ro\big(a(n) + \ldots + a(n-u)\big).
$$
Then $\bA(x)\in\RT_1$.
\end{lemma}
\begin{proof}
By (the proof of) Lemma 4.2 of \cite{BeBu2004}.
\end{proof}

\subsection{The definition of $\fS$}\label{class S}

Let $\fS$ be the set of power series $\bP(x)\in x\cdot \bbN[[x]] \smallsetminus \{0\}$
 that can be expressed in the form
\begin{equation}\label{PS def}
    \bp_0(x) \ +\  \sum_{i=1}^k \bp_i(x)\cdot \bR_i\big(x^{\fd_i}\big), 
\end{equation}
where
$\bp_i(x)\in\bbN[x]$ for $0\le i\le k$,  $\bR_i(x)\in\RTN_1$,
and
$\fd_i\in\bbN\setminus\{0\}$, for $1\le i\le k$.

Let
\begin{eqnarray*}
	\fQ&:=& \Big\{x^\fc\bR\big(x^\fd\big) \in x\cdot \bbN[[x]]\ : \  \fc,\fd\in \bbN, \
	\bR(x)\in\RTN_1,\ 0\le \fc < \fd \Big\}. 
\end{eqnarray*}
Note that every member of $\fS$ can be expressed as a polynomial
from $x\cdot \bbN[x]$ plus a sum of zero or more members of $\fQ$. 

\begin{lemma}\label{prop P2}
Given a member of $\fS$, say
$$
	\bP(x)\ =\ \bp_0(x)\ +\ \sum_{\ell=1}^k x^{\fc_\ell}\bR_\ell\big(x^{\fd_\ell}\big), 
$$
 let $\fd$ be a positive integer divisible by all the exponents $\fd_\ell$
 ($\fd$ can be any positive integer if $\bP(x) = \bp_0(x)$).
Then one can express $\bP(x)$ in the form
\begin{equation} \label{same d}
	\bp_0(x)\ +\ \sum_{i\in I} x^{\fc_i}\bS_i\big(x^{\fd}\big), 
\end{equation}
where  the $x^{\fc_i}\bS_i(x^\fd)$ are in $\fQ$, and $I$ is a 
 finite subset of $\bbN$.

%
\end{lemma}
\begin{proof}
We only need to consider the case that $\bP(x)$ is not a polynomial.
Suppose $x^\fa \cdot \bR\big(x^\fb\big)\in \fQ$ with $\fb \big| \fd$. 
Then
\begin{eqnarray*}
	x^{\fa}\cdot \bR\big(x^{\fb}\big)
&=& 
	x^{\fa}\sum_{n=0}^{\infty}r(n)x^{n \fb}\\
&=& 
	x^{\fa}\sum_{j=0}^{\fd/\fb\,-\,1}\sum_{n=0}^{\infty}
	r \Big(j+\frac{\fd}{\fb}n\Big)x^{j\fb+n\fd}\\
&=& 
	\sum_{j=0}^{\fd/\fb\,-\,1}x^{\fa + j\fb}\sum_{n=0}^{\infty}
	r\Big(j+\frac{\fd}{\fb}n\Big)\big(x^\fd\big)^n \\
&=&
	\sum_{j=0}^{\fd/\fb\,-\,1}x^{\fa + j\fb}\cdot \bS_j(x^\fd)
\end{eqnarray*}
where the $\bS_j(x)$ are the power series in the previous line. 
One easily verifies that each $x^{\fa +j \fb} \cdot \bS_j(x^\fd)$ 
which is non-zero is in $\fQ$. Applying this to each of the 
$x^{\fc_\ell}\bR_\ell\big(x^{\fd_\ell}\big)$ in the expression given for 
$\bP(x)$, and collecting terms based on the lead monomial $x^j$,
 yields the desired result.
%

\end{proof}

 
\begin{lemma}\label{prop P}
\begin{thlist}
\item
$x\cdot \bbN[x] \smallsetminus \{0\} \,\subseteq\,\fS$. 
\item
$x\cdot \RTN_1\ \subseteq\ x\cdot \RTNOS\ \subseteq\ \fS$.
\item
$\fS$ contains the functions $x$, and $x/(1-x^m)$ for $m\ge 1$,  and is closed 
under the operations of scalar multiplication by positive integers, addition and 
Cauchy product.
\end{thlist}
\end{lemma}
\begin{proof}
Properties (a) and (b) are obvious.
For (c), note that the function $x$ is a polynomial, and 
$x/(1-x^m)\in\RTNOS$ since $x/(1-x)\in\RTN_1$; 
so $x,x/(1-x^m)\in\fS$.   
Clearly $\fS$ is closed under scalar multiplication by positive reals, 
and under addition. 
To show that $\fS$ is closed under Cauchy product, take $\bP_1(x),\bP_2(x)\in\fS$  
and express each one in the form \eqref{same d}, using the same $\fd$.
Multiply out the two sums, and note that  $\bR(x), \bS(x) \in \RT_1$ implies
$\bR(x)\cdot \bS(x)\in \RT_1$, by Lemma \ref{closure in RT}, thus 
$\bR(x^\fd)\cdot \bS(x^\fd) \in \RTOS$.
\end{proof}

\subsection{The Star Transformation}

The star transformation on a power series plays an important role in enumeration 
of unlabelled structures and in additive number systems, namely given $\bA(x)$ 
and $\bP(x)$ that satisfy the partition identity \eqref{part id}, one has the well-known 
form
$$
	1 + \bA(x)\ =\ \exp\big(\bP^\star(x)\big),
$$
that was introduced by P\'olya in 1937, where
$$
	\bP^\star(x)\ =\ \sum_{m\ge 1} \bP(x^m) / m,
$$
Writing $\bP^\star(x)$ as 
$\sum_n p^\star(n) x^n$ one has
\begin{eqnarray*}
	p^\star(0) &:=& 0\\
	p^\star(n) &:=& \sum_{jk=n} \frac{p(j)}{k}\ =\ \frac{1}{n}\sum_{d | n} d p(d), 
	\quad\mbox{for }n\geq 1, 
\end{eqnarray*}
We call $\bP^\star(x)$ the {\em star transformation} of $\bP(x)$.

Proposition \ref{combined prop2} below says that 
for $\bP(x) \in \fS$ one has   $\bA(x)\in\RTOS$. 
The proof of this reduces to showing
$\exp\big(\bQ^\star(x)\big) \in \RTOS$ for $\bQ(x)\in\fQ$.
For this we develop properties of an auxiliary function $\whbQ(x)$.

\begin{definition}
Given $\bQ(x)=\sum q(n)x^n$  let
$$
\whbQ(x)\ :=\ \dfrac{x}{1-x}\cdot \dfrac{d}{dx}\bQ^\star(x).
$$
\end{definition}
Thus we have
\begin{equation}\label{nSs}
	nq^\star(n)\ =\ \sum_{d|n}d q(d)\quad\text{for }n\ge 1, 
\end{equation}
and, with $\whbQ(x) = \sum \whq(n) x^n$,
\begin{equation}\label{whq}
	\whq(n)\ =\ q^\star(1)\,+\cdots+\,nq^\star(n) \quad\text{for } n \ge 1. 
\end{equation}

%
\begin{lemma}\label{star lem2}
Suppose $\bQ(x) \in \fQ$.  Then $\whbQ(x) \in \RT_1$.
\end{lemma}
\begin{proof}
We can assume
$\bQ(x) = x^\fc\bR\big(x^\fd\big)$ with 
$\bR(x)\in\RTN_1$, $0\le \fc < \fd$.
$\whq(n)$ is a nondecreasing sequence of nonnegative integers that is 
eventually positive. 
From \eqref{nSs} and \eqref{whq} we have
\begin{equation}\label{expans}
	\whq(n) \ =\  \sum_{m=1}^n m q^\star(m)
	\ =\ \sum_{m=1}^n  \sum_{d|m} d q(d) 
	\ =\  \sum_{m=1}^n \Big\lfloor\frac{n}{m}\Big\rfloor mq(m). 
\end{equation}

Fix $\varepsilon \in (0,1)$ and choose 
\begin{equation}\label{choose M}
	M\ >\ \frac{\fd}{\varepsilon\cdot (1-\varepsilon)}. 
\end{equation}
For any fixed integer $v$, 
\begin{eqnarray*} 
	\frac{\big(\fc+(j-v)\fd\big)\cdot q\big(\fc+(j-v)\fd\big)}{(\fc+j\fd)\cdot q(\fc+j\fd)}
&=&
	 \frac{\big(\fc+(j-v)\fd\big)\cdot r\big(j-v)} {(\fc+j\fd)\cdot r(j)}\\
&\rightarrow&
	1\quad\text{ as }j\rightarrow \infty. 
\end{eqnarray*}
Hence we can choose $N>M^3$ such that
$$
	\big|(\fc+j\fd)\cdot q(c+j\fd)\,-\,\big(\fc+(j-v)\fd\big)\cdot 
	q\big(\fc+(j-v)\fd\big)\big|\ <\ \varepsilon (\fc+j\fd)\cdot q(\fc+j\fd)
$$
for\enspace $0\le v\fd\le M$ and $\fc+j\fd\geq N/M$, and thus
$$
	\big(\fc+(j-v)\fd\big)\cdot q\big(\fc+(j-v)\fd\big)\ >\ 
	(1-\varepsilon)\cdot (\fc+j\fd)\cdot q(\fc+j\fd)
$$
for\enspace $0\le v\fd\le M$ and $\fc+j\fd\geq N/M$\, where $j\ge 0$. 
But then
\begin{equation}\label{bound on s}
	(n-v\fd) q(n-v\fd)\ \ge\  (1-\varepsilon)\cdot n q(n)
\end{equation}
for\enspace $0\le v\fd\le M$ and $n\geq N/M$,  for if $n$ is not of the form 
$\fc + j\fd$ then the right side of \eqref{bound on s} is 0.

For integers $d_1,d_2$ with $1\leq d_1<d_2\leq M$,  and for $n\geq N$,  
we have
$$
\frac{n}{d_1} -\frac{n}{d_2}
\ =\  \frac{n(d_2-d_1)}{d_1d_2}
\ \geq\  \frac{n}{M^2}
\ >\  \frac{M^3}{M^2}
\ =\  M, 
$$
and thus $\displaystyle \frac{n}{d_2} \ <\ \frac{n}{d_1} - M$. 
Consequently, for $n\geq N$, if $d_1<\cdots<d_k $ are the divisors of $n$ that are 
less than $M$, we have
\begin{equation}\label{gap}
	\frac{n}{M} \ <\ \frac{n}{d_k} - M\ <\ 
	\frac{n}{d_k}\ <\ \frac{n}{d_{k-1}} - M\ <\ 
	\cdots\ <\ 
	\frac{n}{d_1} - M \ <\ \frac{n}{d_1} \ = \ n, 
\end{equation}
where the first inequality follows from $d_k < M$ and $n \ge N > M^3$.
Thus the intervals 
$$
	I_d \:=\ \Big[\frac{n}{d} - M, \frac{n}{d}\Big]\quad \text{ for } d < M, d \big| n,
$$
are pairwise disjoint subintervals of $(n/M, \, n]$. 
For $d< M$ and $d\big| n$ we have
\begin{eqnarray}
	\sum_{j \in I_d} jq(j)
&\ge&
	\sum_{\substack{j\in I_d\\ q(j) \neq 0}} (1-\varepsilon) \frac{n}{d} q\Big(\frac{n}{d}\Big) 
	 \quad 	\mbox{by }(\ref{bound on s})\nonumber\\
&\ge&
	\frac{M}{\fd} (1-\varepsilon) \frac{n}{d} q\Big(\frac{n}{d}\Big), \label{sum Id}
\end{eqnarray}
the last inequality following from the fact that $q(j) \neq 0$ implies $j \equiv \fc\, (\mod\ \fd)$.

Returning to the expression for $\whq(n)$ in (\ref{expans}), now assuming
that $n\geq N$, we have
\begin{eqnarray*} 
	\whq(n)
&=& 
	\sum_{j=1}^n \Big\lfloor\frac{n}{j}\Big\rfloor jq(j) \\
&=& 
	\sum_{1\leq j \leq n/M} \Big\lfloor\frac{n}{j}\Big\rfloor jq(j) 
	\ +\ \sum_{n/M<j\leq n} \Big\lfloor\frac{n}{j}\Big\rfloor jq(j) \\
&\ge& 
	M \sum_{1\leq j\leq n/M} jq(j) \ +\  \sum_{n/M<j\le n} jq(j)\\
&\ge & 
	\frac{1}{\varepsilon}\sum_{1\le j\le n/M} jq(j)
	\ +\  \sum_{d|n\atop d< M} \biggl(\sum_{j \in I_d} jq(j)\biggr)
	\quad \mbox{by }(\ref{choose M}),(\ref{gap})\\
&\ge & 
	\frac{1}{\varepsilon}\sum_{{d|n}\atop{d\le n/M}} dq(d)
	\ +\  \sum_{{d|n}\atop {d<M}} \frac{M}{\fd}(1-\varepsilon)\frac{n}{d}\cdot 
      q\Big(\frac{n}{d}\Big)\quad\text{ by }\eqref{sum Id}\\
&\ge & 
	\frac{1}{\varepsilon}\sum_{{d|n}\atop {d\le n/M}} dq(d)
	\ +\ \frac{1}{\varepsilon}\sum_{d|n\atop d>n/M} dq(d)\quad\mbox{by }(\ref{choose M})\\
&=& 
	\frac{1}{\varepsilon}\sum_{d|n} dq(d)\\
&=& 
	\frac{1}{\varepsilon}\big(\whq(n)-\whq(n-1)\big) 
	\quad \mbox{by }(\ref{nSs}). 
\end{eqnarray*}
Thus\enspace $0\,\le\, \whq(n)-\whq(n-1)\, \le\, \varepsilon \whq(n)$,
\enspace so\enspace $\whbQ(x) \in \RT_1$. 
\end{proof}

\subsection{P\'olya Exponentiation}

For $m\ge 1$ the P\'olya exponentiation operators $\sE_m, \sE_{\ge m}$ map 
$x\cdot \bbN[[x]]$ into $x \cdot \bbN[[x]]$, since they convert generating functions  into 
generating functions  for multisets of the original objects.
In this section we will prove that for $m\ge 1$,  $\sE_m$ and $\sE_{\ge m}$ map
 $\fS$ into $\fS$;
and
$\sE_{\ge m}$  maps $\fS$ into $\RTOS$.
%
First we show that the $\sE_m$ map $\fS$ into $\fS$.  

%
\begin{proposition}\label{Em lemma}
Suppose $\bP(x)\in\fS$. Then 
 $$
	 \sE_m\big(\bP(x)\big)\in\fS \quad\text{ for }m\ge 1.
 $$
\end{proposition}
\begin{proof}
By the definition of $\sE_m$ in $\S$\ref{scalar mult constr},  there is a polynomial 
$\bS(y_1,\ldots,y_m) \in \bbQ[y_1,\ldots,y_m]$ 
with nonnegative coefficients such that 
$$
	\sE_m(\bP(x)) \ = \ \bS\Big(\bP(x),\ldots,\bP(x^m)\Big).
$$
Let $M$ be a positive integer such that 
$M\cdot \bS(y_1,\ldots,y_m) \in \bbN[y_1,\ldots,y_m]$.
Since the $\bP(x^i)$ are in $\fS$,  and since $\fS$ is closed under scalar 
multiplication by positive integers, addition and multiplication,  it follows that 
$M\cdot \sE_m\big(\bP(x)\big)$ is also in $\fS$.  
Since $\sE_m(\bP(x))\in \bbN[[x]]$, it follows that $\sE_m\big(\bP(x)\big)\in \fS$.
\end{proof}

With the help of the next lemma we will show that  the $\sE_{\ge m}$ map 
$\fS$ into $\RTOS$. 

%
\begin{lemma}\label{using gcd}
Given a power series $\bA(x)$ and $m$ a positive integer let
$$
	\bA_m(x) \:=\ \Big(\sum_{j=0}^{m-1} x^j\Big)\cdot\bA(x). 
$$
Then
\begin{thlist}
\item
$\bA_m(x) \in \RTN_1$\, implies 
$$
	a(n) - a(n-m)\ =\ \ro\Big(\sum_{j=0}^{m-1} a(n-j)\Big); 
$$
\item
if $\bA_{m_i}(x)\in\RTN_1$\, for $i=1,\ldots,k$\, and 
$d = \gcd\big(m_1,\ldots,m_k)$\, then
$$
	a(n) - a(n-d)\ =\ \ro\Big(\sum_{j=0}^u a(n-j)\Big), 
$$
for a suitable choice of $u$; 
\item
if $d=1$ in \mbox{\rm(b)} then $\bA(x)\in\RTN_1$. 
\end{thlist}
\end{lemma}
\begin{proof}
The definition of $\bA_m(x)$ gives
$$
	a_m(n)\ =\ \sum_{j=0}^{m-1} a(n-j); 
$$
so one has
$$
	a(n)- a(n-m)\ =\ a_m(n) - a_m(n-1). 
$$
Then $\bA_m(x)\in\RTN_1$ gives
$$
	a(n)-a(n-m)\ =\ \ro\big(a_m(n)\big);  
$$
so
$$
	a(n)-a(n-m)\ =\ \ro\Big(\sum_{j=0}^{m-1} a(n-j)\Big). 
$$
This proves (a).

We give details of the proof of (b) for the case $k=2$, the case we will need in the 
proof of Proposition \ref{combined prop2}; the general case is proved in a similar manner. 
From $d=\gcd(m_1,m_2)$ we know that for some integers $q_1,q_2$ 
we have $d=q_1m_1 + q_2m_2$. 
Let $u_i = |q_im_i|$.  We can assume $u_1\ge u_2$; so $d = u_1 - u_2$. 
Then for $n\ge u_1$
\begin{eqnarray*}
	a(n)-a(n-d)
&=&
	\big(a(n)-a(n-u_1)\big) + \big(a(n-u_1) - a(n-d)\big)\\
&=&
	\big(a(n)-a(n-u_1)\big) - \big(a(n-d) - a(n-d - u_2)\big)\\
&=& 
	\ro\Big(\sum_{j=0}^{u_1-1} a(n-j)\Big)\ +\  \ro\Big(\sum_{j=1}^{u_2-1} a(n-d-j)\Big)\\
&=& 
	\ro\Big(\sum_{j=0}^{u_1-1} a(n-j)\Big). 
\end{eqnarray*}
In this case the choice of $u$ for (b) is $u=u_1-1$. 

Item (c) is then an immediate application of Lemma \ref{oSum}.
\end{proof}

\begin{proposition} \label{combined prop2}
Suppose $\bP(x) \in \fS$. Then
$$
	\sE\big(\bP(x)\big)\,\in\, \RTOS.
$$
\end{proposition}
\begin{proof}
By Lemma \ref{prop P2}, $\bP(x)= \bp_0(x)+\sum \bQ_i(x)$ with $\bp_0(x)\in x \cdot \bbN[x]$ 
and the $\bQ_i(x)\in\fQ$. 
Since 
$$
	\sE\Big(\bp_0(x)+\sum \bQ_i(x)\Big) 
	\,=\,  \sE\big(\bp_0(x)\big)\cdot\prod\sE\big(\bQ_i(x)\big), 
$$
it suffices to show that $\sE(\bQ_i(x)\big)\in\RTOS$  since $\sE\big(\bp_0(x)\big)$
 is either 1 (if $p_0(x)=0$), or it is in $\RTOS$ by Theorem 2.48 of \cite{density}.

Let $\bQ(x):= x^\fc\bR\big(x^\fd\big)$ with $\bR(x)\in\RTN_1$. 
To show $\sE\big(\bQ(x)\big)\in\RTOS$ it suffices to consider the case $\gcd(\fc,\fd)=1$. 
As $r(j)\succeq 1$, it follows that $q(\fc + j \fd )\ge 1$ for $j$ sufficiently large, 
say for $j\ge M$.  

Choose any $j\ge M$. 
Then for $n$ such that $(\fc + j \fd ) \big| n$ we have
\begin{equation}\label{harmonic}
    q^\star(n)\ \ge\ \frac{\fc + j \fd }{n} \cdot q(\fc + j \fd )\ \ge \ 
    \frac{\fc + j\fd}{n}\ =\ -\big[ x^{n} \big] \log\big(1 - x^{\fc + j\fd }\big). 
\end{equation}
It follows that there exists a polynomial $\bp_j(x)$ with nonnegative coefficients such that
\begin{equation} \label{poly shift}
     \bQ^\star (x)\ +\ \bp_j(x)\  +\ \log \big(1-x^{\fc + j \fd }\big)
\end{equation}
has nonnegative coefficients. Then
\begin{equation}  \label{add poly}
   \bU_j(x)\:=\ \bigexpp{\bQ^\star (x)\,+\,\bp_j(x)}
\end{equation}
has nonnegative coefficients. 

We will use Lemma \ref{var on schur} with 
\begin{eqnarray*}
	\bA_j(x) &:=& \big(x+\cdots+x^{\fc + j \fd }\big)\cdot\bU_j'(x)\\
	\bB_j(x) &:=& (1-x^{\fc + j \fd })\cdot\bU_j(x)\\
	\bC_j(x) &:=& \whbQ(x)\ +\ x(1-x)^{-1}\bp_j'(x).
\end{eqnarray*}
By differentiating \eqref{add poly} and adjusting polynomial factors one has 
$\bA_j(x) = \bB_j(x)\cdot \bC_j(x)$. 
Since
$\bB_j(x)$ is the exponential of \eqref{poly shift}, 
$\bB_j(x)$ has nonnegative coefficients, and clearly 
$\bC_j(x)$ has nonnegative coefficients. 
Thus $\bA_j(x)$ also has nonnegative coefficients.

The definition of $\bA_j(x)$ says
$$
	a_j(n)\ =\ nu_j(n)\,+\,\cdots\,+\, (n-\fc - j \fd +1)\cdot u_j(n-\fc - j \fd +1),
$$
and from this follows
\begin{eqnarray*}
	a_j(n)\, -\, a_j(n-1) 
&=& 
	nu_j(n) \,-\,(n-\fc - j \fd )\cdot  u_j(n-\fc - j \fd )\\
&=& 
	n \Big( u_j(n) \,-\, u_j(n-\fc - j \fd ) \Big) \ +\ (\fc + j\fd) u_j(n-\fc - j \fd ).
\end{eqnarray*}
Since $\bB_j(x)$ has nonnegative coefficients, 
$u_j(n) - u_j(n - \fc -j\fd)= b(n) \ge 0$, and thus
$$
	a_j(n) - a_j(n-1)\ \ge \ 0.
$$
This shows condition (a) of Lemma \ref{var on schur} holds.

Since $\bB_j(x)$ and $\bU_j(x)$ have nonnegative coefficients,
$$
	0\ \le \ b_j(n)\ \le \ u_j(n) \ \le\  a_j(n)/n, 
$$
and thus $b_j(n) = \ro\big(a_j(n)\big)$. 
This gives condition (b) of Lemma \ref{var on schur}. 


For $n$ larger than the degree of $\bp_j(x)$, we have
\begin{eqnarray}
	c_j(n) 
&=&
	 [x^n]\Big(\whbQ(x)\  +\  x(1-x)^{-1}\cdot\bp_j'(x)\Big)\nonumber\\
&=&
	 \whq(n) \ +\ \bp_j'(1).  \label{eqp1}
\end{eqnarray}
From  \eqref{expans} and \eqref{harmonic} we have
$$
	\whq(n) \ =\ \sum_{m=1}^n mq^\star(m)\ \rightarrow\ 
	\infty\quad \text{as } n \rightarrow \infty,
$$
so (\ref{eqp1}) leads to
\begin{equation}\label{eqp3}
	c_j(n) \ \sim\  \whq(n). 
\end{equation}
By Lemma \ref{star lem2} and Proposition \ref{closure in RT}, 
$\bC_j(x)\in \RT_1$. This is condition (c) of Lemma \ref{var on schur}.

 Now Lemma \ref{var on schur} gives $\bA_j(x)\in \RT_1$, 
that is, 
$$
	\big(1+x+\cdots+x^{\fc + j \fd  -1}\big)\cdot x\bU_j'(x)\in\RT_1. 
$$
Likewise, as $j+1\ge M$, 
$$
	\big(1+x+\cdots+x^{\fc+(j+1)\fd -1}\big)\cdot x\bU_{j+1}'(x)\in\RT_1. 
$$
Without loss of generality we can assume that
	$\bp_j(x)=\bp_{j+1}(x)$;  
let us call this common polynomial $\bp(x)$.  
Then 
	$\bU_j(x) = \bU_{j+1}(x)$; 
let us call this power series simply $\bU(x)$. 

Since $\gcd\big(\fc + j \fd, \fc+(j+1)\fd\big)=1$,  by Lemma \ref{using gcd} (c) 
we have $x\bU'(x)\in \RT_1$. 
But then 
	$\bU(x)\in\RT_1$ 
by Proposition \ref{closure in RT}. From Corollary \ref{schur cor},
$$
	\bigexpp{-\bp(x)}\cdot \bU(x)\in \RT_1, 
$$
that is, 
	$\sE\big(\bQ(x)\big) = \exp\big(\bQ^\star(x)\big)\in\RT_1$,  
proving the proposition.
\end{proof}

\begin{proposition}\label{sE prop}
Suppose\, $\bP(x)\in\fS$. Then 
$$
	\sE_{\ge m} \big(\bP(x)\big) \in \RTOS \quad \text{for } m \ge 0.
$$
\end{proposition}
\begin{proof}
Define
\begin{eqnarray*}
	\bA_m(x)
&:=&
	 \sE_m\big(\bP(x)\big)\\
	\bA(x)
&:=&
	 \sE\big(\bP(x)\big). 
\end{eqnarray*}
It suffices to consider the case that $\gcd\Big( n: p(n) > 0 \Big)= 1$,  in which 
case $\bA(x)$ has integer coefficients that are eventually positive. 
By Proposition \ref{combined prop2}, $\bA(x)  \in \RT_1 $; 
and $\bA(1) = \infty$ since $\bA(x)$ has eventually positive integer coefficients. 

Let $m$ be a fixed positive integer.
By Lemma 3.55, p.~69 of \cite{density}, $a_m(n) \ =\ \ro\big(a(n)\big)$.
Since
$$
	\bA_{\ge m} (x)\ =\ \bA(x) - \sum_{j=0}^{m-1} \bA_j(x),
$$
it follows that
$$
	a_{\ge m} (n)\ \sim\ a(n),
$$
showing that $\bA_{\ge m}(x)\in\RT_1$.
\end{proof}

The following proposition collects the main results concerning $\fS$.
%
%
\begin{proposition}\label{main RT results}
\begin{thlist}
\item
$x, \dfrac{x}{1-x^n} \in \fS$, for  $n\ge 1$. 
\item
$\fS$ is closed under  addition, Cauchy product, $\sE_m$ and $\sE_{\ge m}$, for $m\ge 1$.
\item
For $m\ge 0$ and $\bP(x)\in\fS$, one has
$\sE_{\ge m} \big(\bP(x)\big) \in \RTOS$. 
\end{thlist}
\end{proposition}

\section{The Main Result}

\begin{theorem} \label{main thm}
Let $\cF$ be an adequate MSO-class of forests, say $\cF = (\ge 1)(\cT)$.
Then $\cF$ has a MSO 0--1 law iff the radius of $\cF$ is 1 iff the radius of 
$\cT$ is $\ge 1$.
\end{theorem}
\begin{proof}
Since the radius of the class of all trees is positive, we know that the radius 
of $\cF$ must be positive.
 Then, from  Corollary \ref{forest thm}, $\cF$ has a MSO 0--1 law 
 iff $\bF(x) \in\RTOS$. 
$\bF(x)\in\RTOS$ implies $\rho_\bF = 1$, and this implies $\rho_\bT \ge 1$. 
By Corollary \ref{in GEN}, $\rho_\bT \ge 1$ implies $\bT(x)\in\fG$, and then 
Proposition \ref{main RT results} shows that 
$\bF(x) = \sE\big(\bT(x)\big) \subseteq \RTOS$. 
Thus $\bF\in\RTOS$ iff $\rho_\bF=1$ iff $\rho_\bT \ge 1$.
\end{proof}

\begin{remark}
The main theorem, using essentially the same proof, holds in the more general 
setting of forests with finitely many unary predicates.
 \end{remark}
 
 In the study of spectra, one finds that the periodicity results for MSO classes of 
 trees lift to the setting of MSO classes of unary functions (viewed as functional 
 digraphs). 
 This leads to the natural query:
 
 \begin{problem}
 Does every adequate MSO-class $\cU$ of unary functions (with finitely many 
 unary predicates) of radius 1 have a MSO 0--1 law?
 \end{problem}

\section{Appendix: Monadic Second Order Logic}   \label{MSO sec}

The details of MSO logic for the single binary relation symbol $\lle$ given in this 
section are based on the presentation in  Chap.~6 of \cite{density}.
In addition to the symbols $\lle$ and $=$ we have:
\begin{itemize}
\item
symbols for {\em propositional connectives}, say $\neg$ (not), $\wedge$ (and), 
$\vee$ (or), $\rightarrow$ (implies),  $\leftrightarrow$ (iff);  
\item
the {\em quantifier\/} symbols $\univ$ (for all) and $\exist$ (there exists); 
\item
a set $\cX$ of {\em first order variables}; 
\item
a set $\cU$ of {\em monadic second order variables}.
\end{itemize}

The {\em  MSO-formulas} are defined as follows, by induction:
\begin{itemize}
\item 
the {\em atomic formulas} are expressions of the form
\subitem
$x\lle y$, \, $x=y$,   and $U(x)$; 
\item
if $\varphi$ and $\Psi$ are MSO-formulas then so are
$$
	(\neg\,\varphi)\quad (\varphi\vee\Psi) \quad (\varphi\wedge \Psi) 
	\quad (\varphi\rightarrow \Psi) \quad (\varphi \leftrightarrow \Psi); 
$$
\item
if $\varphi$ is a MSO-formula then so are $(\univ x\,\varphi)$, $(\exist x\,\varphi)$, 
$(\univ U\,\varphi)$ and $(\exist U\,\varphi)$.
\end{itemize}
The \emph{MSO-sentences} are the MSO-formulas with no free occurrences 
of variables.
A \emph{MSO-class} is the class of \emph{finite} models of a MSO-sentence.  
For $\fq \ge 0$, the $\MSO^\fq$-classes are the MSO classes defined by a 
MSO-sentence of quantifier rank at most $\fq$.

\end{document}